\documentclass[12pt]{amsart}
\usepackage{amssymb}
\usepackage{amsfonts}
\usepackage{amsmath,latexsym,amssymb,amsfonts,amsbsy, amsthm}
\usepackage[usenames]{color}
\usepackage{xspace,colortbl}
\usepackage{mathrsfs}
\usepackage[colorlinks,linkcolor=blue,citecolor=blue]{hyperref}
\usepackage[titletoc]{appendix}
\allowdisplaybreaks

\newcommand{\beq}{\begin{equation}}
\newcommand{\eeq}{\end{equation}}
\newcommand{\ben}{\begin{eqnarray}}
\newcommand{\een}{\end{eqnarray}}
\newcommand{\beno}{\begin{eqnarray*}}
\newcommand{\eeno}{\end{eqnarray*}}
\newcommand{\R}{\mathbb{R}}
\newtheorem{thm}{Theorem}[section]
\newtheorem{defi}[thm]{Definition}
\newtheorem{lem}[thm]{Lemma}
\newtheorem{prop}[thm]{Proposition}
\newtheorem{coro}[thm]{Corollary}

\def\m{\mu}
\def\R{\mathbb{R}}

\def\N{\mathbb{N}}
\def\H{\mathcal{H}}
\def\e{\epsilon}
\def\vare{\varepsilon}

\def\D{\Delta}
\def\O{\Omega}

\def\a{\alpha}

\def\g{\gamma}

\def\m{\mu}

\def\de{\partial}

\def\k{\kappa}

\newcommand{\chrc}[1]{\mathrm{1}_{#1}}

\def\Xint#1{\mathchoice
{\XXint\displaystyle\textstyle{#1}}%
{\XXint\textstyle\scriptstyle{#1}}%
{\XXint\scriptstyle\scriptscriptstyle{#1}}%
{\XXint\scriptscriptstyle\scriptscriptstyle{#1}}%
\!\int}
\def\XXint#1#2#3{{\setbox0=\hbox{$#1{#2#3}{\int}$ }
\vcenter{\hbox{$#2#3$ }}\kern-.6\wd0}}

\def\dashint{\Xint-}

\title[Stable solutions to the one-phase free boundary problem]{Nondegeneracy for stable solutions to the one-phase free boundary problem}

\author[N. Kamburov]{Nikola Kamburov$^\ast$}
\address{$^\ast$Facultad de Matem\'aticas \\ Pontificia Universidad Cat\'olica de Chile \\ Avenida Vicu\~na Mackenna 4860, Santiago 7820436, Chile}
\email{nikamburov@mat.uc.cl}

\author[K. Wang]{Kelei Wang$^\dag$}
\address{$^\dag$School of Mathematics and Statistics \\ Wuhan University\\
	Wuhan 430072, China}
\email{wangkelei@whu.edu.cn}

\thanks{ N. Kamburov was partially supported by Proyecto Fondecyt Regular No.\ 1201087. K. Wang  was supported by the National Natural Science Foundation of China (No.~11871381 and No. 12131017).}

\begin{document}

\begin{abstract}
We prove  the nondegeneracy condition for stable solutions to the one-phase free boundary problem. The proof is by a De Giorgi iteration, where we need  the Sobolev inequality of Michael and Simon  and, consequently, an integral estimate for the mean curvature of the free boundary. %This estimate is obtained by using  the stability condition \emph{twice}.
We then apply the nondegeneracy estimate to obtain local curvature bounds for stable free boundaries in dimension $n$, provided the Bernstein-type theorem for stable, entire solutions in the same dimension is valid. In particular, we obtain this curvature estimate in $n=2$ dimensions. 
\end{abstract}


\bibliographystyle{plain}

\maketitle

%%%%%%%%%%%%%%%%%%%%%%%%%%%%%%%%%%%%%%%%%%%%%%

%\tableofcontents

\section{Introduction}

In this article we study stable solutions to the one-phase free boundary problem
 \begin{equation}\label{eqn}
\left\{
\begin{aligned}
 u\geq 0 & \quad \mbox{in } D, \\
   \Delta  u=0 & \quad \mbox{in } D^+(u):=\{x \in D: u(x)>0\}, \\
   |\nabla u|=1 & \quad \mbox{on }  F(u):=\de D^+(u) \cap D.
  \end{aligned}\right.
  \end{equation}
Here, $D\subset \R^n$ is a domain, $D^+(u)$ denotes the \emph{positive phase} of $u$  in $D$ and $F(u)$ is its \emph{free boundary}. We will be mainly interested in \emph{classical} solutions, that is: $u$ is continuous in $\overline{D}$, the free boundary $F(u)$ is a $C^\infty$ hypersurface with $u=0$ on one side, $u>0$ on the other, and $u\in C^\infty(\overline{D^+(u)})$ satisfies the gradient condition $|\nabla u| = 1$ in pointwise sense, evaluated from the positive side. 

The one-phase free boundary problem (FBP) arises as the Euler-Lagrange equations for the energy functional
\begin{equation}\label{functional}
    J(v,D) = \int_{D} (|\nabla v|^2 + 1_{\{v>0\}}) \, dx,  \qquad
    v :D \to[0,\infty),
\end{equation}
and appears in various interface models in fluid mechanics and materials science. There is vast literature on it -- see for example  the books \cite{CafSalsa} and \cite{FriedBook}. A seminal role in its study is played by the 1981 paper \cite{Alt-Caffarelli} of Alt and Caffarelli on minimizers of $J$, in which the authors pioneered the use of blow-up limits to investigate the regularity of the free boundary. In the context of the one-phase FBP, the blow-up analysis is rooted in two basic estimates. For the sake of simplicity, we shall state them for solutions of \eqref{eqn} defined in the unit ball $D=B_1(0)$. The first fundamental estimate is the uniform Lipschitz bound that \emph{any} solution to \eqref{eqn} in $B_1$ satisfies: if $0\in F(u)$, then
\begin{equation}\label{eq:Lip}
|\nabla u(x)| \leq C \quad \text{for all } x\in B_{1/2}^+(u),
\end{equation}
for some positive constant $C$ that depends only on the dimension $n$. The second one is the uniform \emph{nondegeneracy} bound:
\begin{equation}\label{eq:nondegintro}
\dashint_{\de B_r(p)} u \, d\H^{n-1} \geq c r \quad \text{for all } p\in F(u)\cap B_{1/4} \quad \text{and all } r\in (0,1/4),  
\end{equation}
for some dimensional constant $c>0$. Though valid for one-phase energy minimizers (\cite{Alt-Caffarelli}), the nondegeneracy condition \eqref{eq:nondegintro} does not hold for all solutions, as exhibited by the  family of solutions $\{u_\vare\}_{\vare\in (0,1)}$ in $B_1$,
\[u_\varepsilon(x)=
\begin{cases}
	\varepsilon \left(\log\frac{|x|}{\varepsilon}\right)^+, & \text{if} ~ n=2,\\
	 \frac{\varepsilon}{n-2} \left[1-\left(\frac{|x|}{\varepsilon}\right)^{2-n}\right]^+, & \text{if} ~ n\geq 3, 
\end{cases}\]
when $\vare>0$ is small enough.

In their paper \cite{Alt-Caffarelli} Alt and Caffarelli used a convenient measure-theoretic reformulation of \eqref{eq:Lip} and \eqref{eq:nondegintro}. Denote
\[
\mu:=\mathcal{H}^{n-1}\lfloor_{F(u)}.
\]
The Lipschitz bound \eqref{eq:Lip} means that for some dimensional constant $C>0$,
\begin{equation}\label{upper bound}
	\mu(B_r(x))\leq Cr^{n-1} \quad \text{for all } x\in F(u)\cap B_{1/4}, \quad\text{and all } r\in(0,1/4).
\end{equation}
This upper measure estimate can be easily deduced by applying the Divergence Theorem to $\D u$ in $B_1^+(u)\cap B_r(x)$, and using \eqref{eq:Lip} as well as the free boundary condition, $|\nabla u| = 1$ on $F$. In turn,  the lower estimate
\begin{equation}\label{lower bound}
	\mu(B_r(x))\geq c r^{n-1} \quad \text{for all } x\in F(u)\cap B_{1/4}, \quad\text{and all } r\in(0,1/4),
\end{equation}
where $c>0$ is a dimensional constant, corresponds to the nondegeneracy condition \eqref{eq:nondegintro}. The equivalence of the bounds \eqref{eq:Lip}-\eqref{eq:nondegintro} to their measure-theoretic counterparts \eqref{upper bound}-\eqref{lower bound} is the statement of \cite[Theorem 4.3]{Alt-Caffarelli}.

We will be interested in solutions $u$ of \eqref{eqn} which are \emph{stable} critical points of the functional $J$ with respect to compactly supported domain deformations:
\[
\left.\frac{d^2}{dt^2}\right|_{t=0} J (u(x+t \Phi(x)), D) \geq 0 \quad \text{for all vector fields } \Phi\in C^\infty_0(D,\R^n).
\]
For classical solutions, this condition takes the form of the so called \mbox{\emph{stability inequality}} (see \cite{CJK} for its derivation):% (see \cite{CJK})
\begin{equation}\label{stability}
	\int_{F(u)} H\phi^2 \, d\m \leq \int_{\O}|\nabla\phi|^2 \, dx, \quad \text{for all } \phi\in C^{\infty}_0(D),
\end{equation}
where $H$ denotes the \emph{mean curvature} of the free boundary with respect to the inner unit normal vector to the positive phase $\O:=D^+(u)$. % (see \cite{CJK} for the derivation).
The principal goal of this paper is to show that the nondegeneracy condition \eqref{lower bound} (equivalently, \eqref{eq:nondegintro}) holds, more generally, for stable solutions.

\begin{thm}\label{thm nondegeneracy}
  There exists a  constant $\varepsilon(n)$  such that if $u$ is a stable classical solution of \eqref{eqn} in $B_1$, then for any $x\in F(u)\cap B_{1/4}$ and $r\in(0,1/4)$,
\[\mu(B_r(x))\geq \varepsilon(n) r^{n-1}.\]
\end{thm}

Alt and Caffarelli \cite{Alt-Caffarelli} showed that the nondegeneracy condition \eqref{eq:nondegintro} is valid for energy minimizing solutions to the one-phase FBP. Soon thereafter, it was also established for energy minimizers within the theory of two-phase FBP \cite{ACF}.  More recently, non-degeneracy bounds have underlied the studies of the regularity theory for energy minimizing solutions to the ``thin" one-phase FBPs (\cite{allen2012two,caffarelli2010variational,  de2012regularity, de2015regularity}); for ``almost minimizers" of the one-phase functional $J$  (see \cite{ david2019free,david2015regularity, de2020almost}), as well as for energy minimizing solutions to a vectorial analogue of \eqref{eqn} (see \cite{caffarelli2018minimization, mazzoleni2017regularity, mazzoleni2020regularity}). In all of these, the estimate is obtained via energy comparison methods. In works on free boundary solutions that are \emph{higher} critical points of their underlying energy functionals, the nondegeneracy bound has been achieved either by viewing the solution as a constrained minimizer, as in \cite{Jerison-Perera, liu2021smooth}, and employing energy comparisons again, or by utilizing certain given topological constraints on the free boundary (\cite{JK16}).

None of these methods are available in our setting. We obtain the result of Theorem \ref{thm nondegeneracy} by performing a De Giorgi iteration that takes place on the free boundary surface itself. The stability inequality is used to obtain an $L^1$ estimate of the mean curvature $H$ in terms of the area of the free boundary, contained in a slightly larger scale (Lemma \ref{lem estimate for mean curvature}). We then employ the Sobolev inequality of Michael and Simon \cite{M-Simon} to bound the area in a smaller scale and close the iteration loop. The proof of the key mean curvature estimate, Lemma \ref{lem estimate for mean curvature}, is delicate and involves plugging in the stability inequality a novel test function based on the gradient of $v=h-u$, where $h$ is the harmonic replacement of the solution $u$ in $B_1$ -- rather than the more commonly used, natural test functions based on $|\nabla u|$. % The properties 

The fact that stable solutions enjoy the nondegeneracy estimate opens up the toolbox of blow-up techniques, with which to pry the geometry of the free boundary. The second result of our paper involves obtaining interior curvature bounds for stable one-phase free boundaries, \emph{provided that} the global problem in the same dimension is rigid. More precisely, we show that if one knows that the \emph{Bernstein type} statement
\begin{equation}\label{eq:FBBern}
\text{if } U \text{ is a classical stable solution of \eqref{eqn} in } \R^n \quad \Longrightarrow \quad F(U) \text{ is flat} 
\end{equation}
holds in some dimension $n$, then stable free boundaries in $B_1\subset \R^n$ have uniformly bounded curvature on a smaller scale. This is expressed in the following theorem. %An important role will be played by the non-degeneracy property established in Theorem \ref{thm nondegeneracy}. %It will ensure that the 

\begin{thm}\label{thm:bernstab}
Assume that \eqref{eq:FBBern} is true for some $n\geq 2$, and let $u$ be a stable classical solution to the one-phase FBP \eqref{eqn} in $B_1\subset \R^n$, with $0\in F(u)$. Then there exists a constant $C=C(n)$ such that the second fundamental form $A$ of the free boundary $F(u)$ satisfies
\begin{equation}\label{thm:bernstab:main}
|A(p)| \text{dist}(p, \de B_{1/4}) \leq C \quad \text{for all } p\in F(u)\cap B_{1/4}. 
\end{equation}
In particular, the curvature $|A|$ of $F(u)\cap B_{1/8}$ is bounded by a dimensional constant. 
\end{thm}
We remark that the free boundary curvature bound \eqref{thm:bernstab:main} means that for some dimensional constant $c>0$, the connected component of $B_c^+(u)$, whose boundary contains the origin, is the supergraph of a function $f$ with bounded $C^2$-norm. By the classical regularity theory for the one-phase FBP \cite{KindNiren}, one then gets higher order derivative estimates for $f$.  

Theorem \ref{thm:bernstab} is the free boundary analogue of a well known relation in minimal surface theory: between Bernstein-type results for complete, stable minimal hypersurfaces on the one hand, and local curvature estimates for stable minimal hypersurfaces  on the other (see \cite[Chapter 3]{white2013lectures}). As such, its proof employs a similar compactness argument, which in our setting is made possible because of the non-degeneracy estimate of Theorem \ref{thm nondegeneracy}.

The problem of finding the dimensions $n$, for which the rigidity statement \eqref{eq:FBBern} is true, is open and akin to the so called ``Stable Bernstein Problem" for minimal hypersurfaces in $\R^n$ (see \cite{chodosh2021stable}).  The recent construction by  De Silva, Jerison and Shahgholian \cite{de2022inhomogeneous} of entire, energy minimizing, classical solutions of \eqref{eqn} that are asymptotic to the (non-flat) energy minimizing cone of De Silva-Jerison \cite{DeSilvaJerison} in $\R^7$, implies that $n\leq 6$. Note that the stable Bernstein problem  is a little different from the energy minimizing case (see Caffarelli-Jerison-Kenig \cite{CJK} and Jerison-Savin \cite{Jerison-Savin}), because now the blowing down limit of entire solutions could be the wedge solution $|x_n|$.  Entire solutions with such an asymptotic behavior have been constructed in Hauswirth-H\'{e}lein-Pacard \cite{HHP} (in dimension $2$) and Liu-Wang-Wei \cite{liu2021smooth} (in higher dimensions), whose free boundaries are of catenoid type. Although these known examples are unstable, it is not clear in general how to exclude such a possibility by only using the stability condition.

Here we show the veracity of \eqref{eq:FBBern} in dimension $n=2$ (see Theorem \ref{prop:stab} of Section \ref{sec:proofstab}), using the logarithmic cut-off trick. As a corollary, stable free boundaries, defined in disks of $\R^2$, enjoy interior curvature estimates. 

\begin{coro}\label{coro:stab} Let $u$ be a stable classical solution to the one-phase FBP \eqref{eqn} in $B_1\subset \R^2$, and assume $0\in F(u)$. Then there exists an absolute constant $C$ such that the curvature $\k$ of the free boundary $F(u)$ satisfies
\begin{equation}\label{thm:stab:main}
|\k(p)| \text{dist}(p, \de B_{1/4}) \leq C \quad \text{for all } p\in F(u)\cap B_{1/4}. 
\end{equation}
In particular, the curvature $\k$ of $F(u)\cap B_{1/8}$ is bounded by an absolute constant. 
\end{coro}

In the next section we prove Theorem \ref{thm nondegeneracy}. In Section \ref{sec:proofstab} we establish Theorem \ref{thm:bernstab} as well as the Bernstein type Theorem \ref{prop:stab} concerning entire stable solutions of \eqref{eqn} in $\R^2$.  In the Appendix we provide the necessary technical results for the proof of Theorem \ref{thm:bernstab}.

\section{Proof of Theorem \ref{thm nondegeneracy}}

In what follows the letters $C,c$ (possibly with indices and primes) will denote positive constants which depend only on the dimension $n$, and which may change from line to line. The ball of radius $r$, centered at $x\in \R^n$, is denoted by $B_r(x)$ and $B_r:=B_r(0)$. By $\H^k$ we will denote the Hausdorff measure of dimension $k$. For ease of notation, we will often write $F:=F(u)$ and $\O:=B_1^+(u)$. 
%By $B^{n-1}_r(x')$ we will denote the respective ball, centered in $x'\in \R^{n-1}$.

We start with the following auxiliary result about harmonic functions. 
\begin{lem}\label{lem:harmonic}
Let $v$ be a harmonic function in a domain $D\subseteq \R^n$. Then 
\begin{equation}\label{lem:harmonic:eq:2ff}
c |\nabla^2 v|^2 \leq  |\nabla^2 v|^2 - |\nabla |\nabla v| |^2 \leq |\nabla^2 v|^2 \quad \text{a.e.\ in } D,
\end{equation}
for the constant $c=\big(2(n-1)\big)^{-1}$.
\end{lem}
\begin{proof}
If $v =\text{const}$, then \eqref{lem:harmonic:eq:2ff} obviously holds. Assume that $v$ is a nontrivial harmonic function in $D$. %. As is well known, the set of critical points of $v$ locally has finite $\H^{n-2}$ measure
Then $\nabla v(p) \neq 0$ a.e.\ $p\in D$. Fix such a point $p$  and choose a Euclidean coordinate system $(x_1, \ldots, x_n)$ so that the unit vector along $x_n$, $e_n =\nabla v(p)/|\nabla v(p)|$. In this way, $v_i(p) = 0$ for $i\in S:=\{1, 2,\ldots, (n-1)\}$ and $v_n(p) = |\nabla v(p)|$. We then compute at $p$ that $|\nabla |\nabla v|^2|^2 = 4 v_n^2 \sum_j v_{nj}^2$, thereby
\begin{equation}\label{lem:harmonic:eq1}
|\nabla ^2 v|^2 - |\nabla |\nabla v||^2 = \sum_{i, j} v_{ij}^2 - \sum_j v_{nj}^2  = \frac{1}{2} |\nabla^2 v|^2 + \frac{1}{2} \sum_{i,j\in S} v_{ij}^2 - \frac{1}{2}v_{nn}^2.
\end{equation}
Now, since $v$ is harmonic, we have $v_{nn}=-\sum_{i\in S} v_{ii}$ and the AM-GM inequality yields
\[
\frac{v_{nn}^2}{n-1} \leq \sum_{i\in S} v_{ii}^2.  
\]
Combining it with \eqref{lem:harmonic:eq1}, we conclude that at $p$,
\begin{align*}
|\nabla ^2 v|^2 - |\nabla |\nabla v||^2 & \geq \frac{1}{2} \left( |\nabla^2 v|^2 + \sum_{i\in S} v_{ii}^2 - v_{nn}^2 \right) \\
& \geq \frac{1}{2} \left(|\nabla^2 v|^2 - \left((1-\frac{1}{n-1}\right)v_{nn}^2 \right) \geq \frac{1}{2(n-1)}|\nabla^2 v|^2. 
\end{align*}
\end{proof}

We now formulate and prove a key integral estimate for the mean curvature of stable free boundaries. 

\begin{lem}\label{lem estimate for mean curvature}
Let $u$ be a stable classical solution of \eqref{eqn} in  $B_1\subset \R^n$, $n\geq 3$, which satisfies $|\nabla u|\leq C$ in $B_1^+(u)$.   Assume that $0\in F$. There exist two universal  constants $\varepsilon_1$ (small) and $C_1$ (large) such that if
\[ \mu(B_1)=\varepsilon\leq \varepsilon_1,\]
then for any $r\in( C_1\varepsilon^{\frac{1}{n(n-1)}},1)$,
\begin{equation}\label{mean curvature bound}
  \int_{F\cap B_{1-r}}|H| \, d\m \leq C_1r^{-2}\varepsilon^{\frac{n}{n-1}}.
\end{equation}
\end{lem}
\begin{proof}
We divide the proof into four steps.	
	
	{\bf Step 1. A decomposition of $u$.}

Write $u=h-v$, where $h$ is the harmonic extension of the boundary values of $u$ on $\partial B_1$ to $B_1$, while $v$ solves
\begin{equation}\label{eqn for v}
  \left\{\begin{aligned}
 -\Delta v =\mu  & \quad \mbox{in } B_1,\\
v=0 & \quad \mbox{on } \partial B_1.
\end{aligned}\right.
\end{equation}
Since $u$ is Lipschitz continuous on $\de B_1$, standard elliptic regularity theory implies that $h\in C^{1/2}(\overline{B_1})\cap C^{\infty}(B_1)$. Thus,  $v\in C^{1/2}(\overline{B_1})$ and is locally Lipschitz continuous in $B_1$. In fact,  $v$ is smooth in $\overline{\Omega}$ and $\overline{B_1\setminus \Omega}$, but its gradient jumps when crossing the free boundary $F$.

We claim that
  \begin{equation}\label{sup bound}
  \sup_{B_1}v\lesssim \varepsilon^{\frac{1}{n-1}}.
\end{equation}
\begin{proof}[Proof of \eqref{sup bound}]
Let $v^\ast$ be  the solution of the Poisson equation
\[-\Delta v^\ast= \mu\lfloor_{B_1}, \quad \mbox{in } \R^n\]
given by the Newtonian potential
\[ v^\ast(x)= c_n\int_{B_1}|x-y|^{2-n}d\mu(y), \quad \forall x\in\R^n.\]
Take $\rho>0$, and divide the above integral into two parts
\[
v^\ast(x) = \int_{B_1\cap B_\rho(x)} + \int_{B_1\setminus B_\rho(x)} =: \mathrm{I} + \mathrm{II}.
\]
%the definition of $v^\ast$ into two parts, the first one $\mathrm{I}$ on the ball $B_\rho(x)$ and the other one $\mathrm{II}$ on $B_1\setminus B_\rho(x)$.
For the first one, we get by  \eqref{upper bound} that
\begin{eqnarray}\label{L infty from neighborhood}
  \mathrm{I} &\lesssim & \sum_{k=0}^{+\infty} \left(2^{-k}\rho\right)^{2-n} \mu\left( B_{2^{-k}\rho}(x) \right) \nonumber \\
   &\lesssim& \sum_{k=0}^{+\infty}  2^{-k}\rho \\
   &\lesssim & \rho. \nonumber
\end{eqnarray}
In the second part we have $|x-y|>\rho$, so
\[ \mathrm{II}\lesssim \rho^{2-n}\mu(B_1).\]
Combining these two estimates by choosing $\rho:=\mu(B_1)^{1/(n-1)}$, we get
\[ \sup_{\R^n}v^\ast \lesssim \mu(B_1)^{\frac{1}{n-1}}.\]
  By the maximum principle, $v\leq v^\ast$ in $ B_1$,
and \eqref{sup bound} follows.
\end{proof}
The first  consequence of \eqref{sup bound} is that
\begin{equation}\label{energy for v}
	\int_{B_1}|\nabla v|^2 \, dx =-\int_{B_1}v\Delta v\leq C\varepsilon^{\frac{n}{n-1}},
\end{equation}
after an integration by parts using \eqref{eqn for v} and the bound \eqref{sup bound}. 

The second consequence is that the derivatives of $h$ on the free boundary $F$ are small. Indeed, since $u=0$ on $F$, we have 
\[ h(p)=v(p)\leq C\varepsilon^{\frac{1}{n-1}} \quad \text{for all } p\in F.\]
Furthermore, since $h$ is positive and harmonic in $B_r(p)\subset B_1$, for $r=d(p, \de B_1)$, Harnack's inequality tells us that
%\[ \sup_{B_{1-r/4}}h\leq Cr^{1-n}\varepsilon^{\frac{1}{n-1}}.\]
\begin{equation}\label{eq:h}
h(x)\leq c h(p) \leq c_1 \vare^{1/(n-1)} \quad \text{for all } x\in B_{r/2}(p).
\end{equation}
Thus, by interior derivative estimates for harmonic functions, we get
\begin{equation}\label{gradient estimate for h}
	|\nabla h|(p) \leq\frac{C \varepsilon^{\frac{1}{n-1}}}{r}, \quad |\nabla^2 h|(p) \leq \frac{C\varepsilon^{\frac{1}{n-1}}}{r^2} \quad \text{for } p\in F \text{ and } r=d(p,\de B_1).
\end{equation}
As $|\nabla u|=1$ on $F$, \eqref{gradient estimate for h} implies that, %if $r\geq C_1\varepsilon^{1/(n-1)}$ for some large constant $C_1$ (this condition is non-void because $\varepsilon$ is assumed to be sufficiently small), then
\begin{align}
& \inf_{F\cap B_{1-r/2}} |\nabla v|\geq 1-2Cr^{-1}\varepsilon^{\frac{1}{n-1}} \geq 3/4,  \label{lower gradient bound} \\
& \sup_{F\cap B_{1-r/2}} |\nabla v|\leq  1+2Cr^{-1}\varepsilon^{\frac{1}{n-1}} \leq 5/4, \label{upper gradient bound}
\end{align}
whenever $r\geq C_1\varepsilon^{1/(n-1)}$, where $C_1$ is sufficiently large and $\vare\leq \vare_1$ sufficiently small. 

\medskip

{\bf Step 2. Integral estimate on $H^{-}$}.

Let us decompose the mean curvature $H$ of the free boundary $F$ into its positive and negative parts:
\[
H = H^+ - H^-.
\]
In this second step we will obtain the following integral estimate for $H^-$:
\begin{equation}\label{eq:Hneg}
\int_{F \cap B_{1-r}} H^- \, d\m \leq C \e^{n/(n-1)} r^{-2}, \quad \text{provided } r\geq (C_1 \vare^{1/(n-1)})^{1/n}.
\end{equation}
For the purpose, observe that if $\nu$ denotes the outer unit normal to the positive phase $\O$ and $$w:=\frac{1}{2}(|\nabla u|^2 -1),$$ then $H=w_\nu$ on $F$ and 
\[
-H^- = - \de_\nu [w^+] \quad \text{on } F.
\]
Take $\eta\in C_0^\infty(B_{1-r/2})$ to be a standard cut-off function, such that $0\leq \eta \leq 1$, $\eta\equiv 1$ in $B_{1-r}$ and $|\nabla \eta|\leq 4/r$. 
Applying the Divergence Theorem to $\eta^2\nabla[\rho_\e(w)]$ in $\O$, where $\rho_\e(t)$ is a standard, convex regularization of $t^+$, we get after taking the $\lim_{\e\to 0}$:
\begin{align}\label{eq:calcHneg}
\int_{F} H^- \eta^2 \, d\m  & \leq  -\int_{\{|\nabla u|>1\}} \eta^2 \D w \, dx - \int_{\{|\nabla u|>1\}} \nabla \eta ^2 \cdot \nabla w \, dx  \notag \\ %+ \int_{\de \{|\nabla u|>1\}  \cap \O} (|\nabla u|^2-1)_\nu  \eta^2 \, d\m \notag \\
& = -\int_{\{|\nabla u|>1\}} |\nabla^2 u|^2\eta ^2 \, dx  - \int_{\{|\nabla u|>1\}} 2|\nabla u| \eta \left(\nabla |\nabla u| \cdot \nabla \eta\right) \, dx \notag \\
& \leq - \int_{\{|\nabla u|>1\}} \left(|\nabla^2 u|^2 - |\nabla |\nabla u||^2 \right)\eta^2\, dx + \int_{\{|\nabla u|>1\}}  |\nabla u|^2 |\nabla\eta|^2 \, dx, \notag \\
& \leq C \int_{\{|\nabla u|>1\}} |\nabla \eta|^2 \, dx \leq C r^{-2} |\{|\nabla u|>1\}\cap B_{1-r/2}| \notag \\
& \leq  C r^{-2} |\{|\nabla v|> 1- |\nabla h|\}\cap B_{1-r/2}|.
\end{align}
In the computation above, we have used the facts that: $\D |\nabla u|^2 = 2 |\nabla^2 u|^2$ in $\O$ on account of $u$ being harmonic in $\O$, that $\left(|\nabla^2 u|^2 - |\nabla |\nabla u||^2 \right)\geq 0$, and that $|\nabla u|\leq c$.

Now, since $u(0)=0$, we have $h(0)=v(0)\leq c \vare^{1/(n-1)}$, so that Harnack's inequality yields
\[
0\leq h \leq c_0 h(0) r^{1-n} \leq c \vare^{1/(n-1)}  r^{1-n} \quad \text{in } B_{1-r/4}. 
\]
Hence, by interior derivative estimates
\begin{equation}\label{eq:estDh}
|\nabla h| \leq C \vare^{1/(n-1)}  r^{-n} \leq \frac{1}{2}  \quad \text{in } B_{1-r/2} \quad \text{as long as } r\geq (C_1 \vare^{1/(n-1)})^{1/n},
\end{equation}
for some large absolute constant $C_1$.  Combining \eqref{eq:estDh} with the estimate in \eqref{eq:calcHneg}, we obtain by the Chebyshev inequality the desired bound:
\begin{align*}
\int_{F\cap B_{1-r}} H^- \, d\m  & \leq Cr^{-2}|\{|\nabla v|> 1- |\nabla h|\}\cap B_{1-r/2}| \leq Cr^{-2} |\{|\nabla v|> 1/2 \}\cap B_{1-r/2}| \\
& \leq 4C r^{-2}\int_{B_1} |\nabla v|^2 \leq C' \vare^{n/(n-1)} r^{-2}, \quad \text{provided } r\geq (C_1 \vare^{1/(n-1)})^{1/n},
\end{align*}
where the last inequality above follows from \eqref{energy for v}.

\medskip

{\bf Step 3. Using the stability condition.}

Fix $r \in (C_1\varepsilon^{\frac{1}{n(n-1)}},1)$ and take the same cut-off function $\eta\in C_0^\infty(B_{1-r/2})$ %such that $\eta\equiv 1$ in $B_{1-r}$ and $|\nabla \eta|\leq 4/r$. 
from Step 2. Plugging $|\nabla v|\eta$ into the stability inequality \eqref{stability} and integrating by parts, we obtain
\begin{eqnarray}\label{stability with v}
% \nonumber % Remove numbering (before each equation)
  \int_{F\cap B_1}H|\nabla v|^2\eta^2  \, d\m &\leq & \int_{\Omega\cap B_1}|\nabla\left(|\nabla v|\eta\right)|^2 \, dx \nonumber \\
   &=& \int_{\Omega\cap B_1}|\nabla v|^2|\nabla\eta|^2+|\nabla|\nabla v||^2\eta^2+2|\nabla v|\eta \nabla|\nabla v|\cdot\nabla\eta \, dx \nonumber\\
      &=& \int_{\Omega\cap B_1}|\nabla v|^2|\nabla\eta|^2+|\nabla|\nabla v||^2\eta^2-\frac{1}{2}\Delta|\nabla v|^2\eta^2 \, dx \\
      &&+\frac{1}{2}\int_{F\cap B_1}\partial_\nu|\nabla v|^2 \eta^2 \, d\m \nonumber\\
      &=& \int_{\Omega\cap B_1}|\nabla v|^2|\nabla\eta|^2-\left(|\nabla^2 v|^2-|\nabla|\nabla v||^2\right)\eta^2\, dx  \nonumber\\
      &&+\frac{1}{2}\int_{F\cap B_1}\partial_\nu|\nabla v|^2 \eta^2 \, d\m. \nonumber
\end{eqnarray}
%where we used the fact that $\D v = 0$ in $\O$, so that $\D |\nabla v|^2 = 2 |\nabla ^2 v|^2$ in $\O$.

To compute $\partial_\nu |\nabla v|^2$ along $F$, we fix a point $p$ on the free boundary. In suitable Euclidean coordinates $(x^\prime, x_n)\in \R^{n-1}\times \R$ near $p=(0^\prime,0)$, the free boundary $F$ is locally given by
the graph 
$$\{x_n=g(x^\prime)\}, \quad x^\prime\in B_\rho^{n-1}(0^\prime),$$ 
where $g(0^\prime)=0$ and $\nabla^\prime g(0^\prime)=0$. Here we assume $u>0$ in $B_\rho\cap\{x_n>g(x^\prime)\}$. Hence $u_n(0)=1$, $\nabla^\prime u(0)=0$ and $\nu(0)=-\nabla u(0)=-e_n$. Then by differentiating the free boundary condition and utilizing the fact that $\Delta u=0$, we get
\begin{equation}\label{second derivatives of u}
  u_{in}(0)=0, \quad u_{nn}(0)=-\Delta^\prime u(0)=-H(0).
\end{equation}

In the following, we use  summation convention over repeated indices, where  $i\in\{1,\ldots, n\}$.
Under these assumptions, at the origin it holds that
\begin{eqnarray*}
 \frac{1}{2} \partial_\nu|\nabla v|^2&=& -\nabla^2 v(\nabla v,\nabla u)=-v_{i n} v_i   \\
  &=&  -\left(h_{in}-u_{in}\right)\left(h_i -u_i\right)  \\
  &=& -h_{i n}h_i-u_{i n}u_i+h_{i n}u_i+u_{i n} h_i\\
   &=&-u_{nn}+u_{nn} h_n-h_{i n}h_i+h_{nn}  = u_{nn}v_n - h_{i n}h_i+h_{nn}\\
   &=& -H \nabla v\cdot\nabla u+O\left(r^{-3}\varepsilon^{\frac{2}{n-1}} + r^{-2} \varepsilon^{\frac{1}{n-1}} \right)\\
&=&  -H \nabla v\cdot\nabla u+O\left(r^{-2} \varepsilon^{\frac{1}{n-1}} \right) \quad \text{on } F\cap B_{1-r/2}, \text{ for } r\geq C_1\vare^{\frac{1}{n-1}}, %\quad \mbox{(by \eqref{gradient estimate for h} and \eqref{second derivatives of u})} 
\end{eqnarray*}
where the penultimate line follows from  \eqref{second derivatives of u} and the interior estimates \eqref{gradient estimate for h}.
Plugging this formula back into \eqref{stability with v}, we get
\begin{eqnarray*}
% \nonumber % Remove numbering (before each equation)
 &&\int_{\Omega\cap B_1} \left(|\nabla^2 v|^2-|\nabla|\nabla v||^2\right)\eta^2 \, dx \\
 && \leq \int_{\Omega}|\nabla v|^2|\nabla\eta|^2 \, dx -\int_{F}H\left(|\nabla v|^2+\nabla v\cdot\nabla u\right)\eta^2 \, dx + O\left(r^{-2}\varepsilon^{\frac{1}{n-1}}\right)\int_{F}\eta^2 \, d\m.
\end{eqnarray*}
Using \eqref{gradient estimate for h}, we obtain for $r\geq C_1 \vare^{1/(n-1)}$,
\[
|\nabla v|^2+\nabla v\cdot \nabla u = \nabla v \cdot \nabla h = (-\nabla u + \nabla h)\cdot \nabla h =O\left(r^{-1}\varepsilon^{\frac{1}{n-1}}\right) \quad \text{on } F\cap B_{1-r/2}.\]
Therefore
\begin{eqnarray*}
	&&\int_{\Omega\cap B_1}  \left(|\nabla^2 v|^2-|\nabla|\nabla v||^2\right)\eta^2 \, dx \\
	& \leq & \int_{\Omega\cap B_1}|\nabla v|^2|\nabla\eta|^2 \, dx +Cr^{-2}\varepsilon^{\frac{1}{n-1}}\int_{F\cap B_1}\left(1 + |H| r \right) \eta^2\, d\m.
\end{eqnarray*}
Now the result of Lemma \ref{lem:harmonic} allows us to bound the Hessian of $v$:
%\[ |\nabla^2 v|^2\leq c\left(|\nabla^2 v|^2-|\nabla|\nabla v||^2\right) \footnote{needs reference or short proof}.\]
%This implies that
\begin{equation}\label{1}
	\int_{\Omega\cap B_1} |\nabla^2 v|^2\eta^2 \, dx \leq c\int_{\Omega\cap B_1}|\nabla v|^2|\nabla\eta|^2 \, dx +Cr^{-2}\varepsilon^{\frac{1}{n-1}}\int_{F\cap B_1}\left(1 + |H| r \right) \eta^2\, d\m.
\end{equation}

\medskip

{\bf Step 4. Using the stability condition again.}
Now we start from the second line in \eqref{stability with v} and then use \eqref{1} to get
\begin{eqnarray}\label{second use of stability with v}
\int_{F\cap B_1}H|\nabla v|^2\eta^2 \, d\m
    &\leq & C  \int_{\Omega\cap B_1}|\nabla v|^2|\nabla\eta|^2+|\nabla^2 v|^2\eta^2 \, dx \nonumber \\
    &\leq & C\int_{\Omega\cap B_1}|\nabla v|^2|\nabla\eta|^2 \, dx +Cr^{-2}\varepsilon^{\frac{1}{n-1}} \int_{F\cap B_1}\left(1+|H|r\right)\eta^2 \, d\m. 
\end{eqnarray}
Adding $2 \int_{F} H^{-} |\nabla v|^2 \eta^2 \, d\m$ to both sides of \eqref{second use of stability with v}, we obtain
\begin{align}\label{eq:stabtycont}
\int_{F}|H| |\nabla v|^2\eta^2 \, d\m & \leq  C\int_{\Omega}|\nabla v|^2|\nabla\eta|^2 \, dx +Cr^{-2}\varepsilon^{\frac{1}{n-1}} \int_{F} \left(1+|H|r\right)\eta^2 \, d\m \\
& + 2\int_{F} H^- |\nabla v|^2\eta^2 \, d\m. \notag
\end{align}
If $r\geq (C_1\varepsilon^{1/(n-1)})^{1/n} \geq C_1\varepsilon^{1/(n-1)}$, where $C_1$ is sufficiently large, then the gradient bounds \eqref{lower gradient bound}-\eqref{upper gradient bound} imply
\[\frac{1}{2}\leq |\nabla v|^2 - Cr^{-1} \vare^{1/(n-1)} \leq 2 \quad \mbox{on} ~~ F\cap B_{1-r/2}.\]
Thus, collecting all the integral terms involving $|H|$ on the left-hand side, \eqref{eq:stabtycont} becomes
\[\int_{F}|H|\eta^2 \, d\m \leq C\int_{\Omega}|\nabla v|^2|\nabla\eta|^2 \, dx +Cr^{-2}\varepsilon^{\frac{1}{n-1}} \int_{F}\eta^2 \, d\m + C \int_F H^{-} \eta^2. \]
Now, putting together the bounds \eqref{energy for v} of Step 1 and \eqref{eq:Hneg} of Step 2,  we finally obtain the integral bound \eqref{mean curvature bound}.
\end{proof}

\begin{proof}[Proof of Theorem \ref{thm nondegeneracy}]

We will only prove the $n\geq 3$ case of Theorem \ref{thm nondegeneracy}. If $n=2$, the solution of \eqref{eqn} in $B_1\subset \R^2$ can be viewed as  a solution in $B_1\times \R \subset\R^3$  by adding a redundant variable, which preserves the stability condition.

Fix $p\in F\cap B_{1/4}$ and $r\in (0,1/4)$.  After recentering and rescaling, 
$$u \rightarrow \tilde{u}(x):=u(p+rx)/r$$ we may assume that we are dealing with a Lipschitz continuous, stable classical solution $u$, defined in $B_1$, with $0\in F$.  We use a De Giorgi type iteration to prove that if $\varepsilon:=\mu(B_1)$ is small enough (to be chosen later in the proof), then $\mu(B_{1/2})=0$, leading to a contradiction.

Rewriting the estimate \eqref{mean curvature bound} for scales $1/4<R_1<R_2$, where $1-r=R_1/R_2$, we get
\begin{equation}\label{eq:MCB}
\int_{F\cap B_{R_1}} |H| \, d\m \leq \frac{C_1\mu(B_{R_2})^{\frac{n}{n-1}}}{(R_2-R_1)^2},%  \quad \text{provided } 
\end{equation}
provided
\[
R_2- R_1\geq C_1 R_2^{1-\frac{1}{n}} \mu(B_{R_2})^{\frac{1}{n(n-1)}} \geq (C_1/4) \mu(B_{R_2})^{\frac{1}{n(n-1)}}. 
\]
Set
 \[r_0=1, \quad a_0=\mu(B_1)\]
 and for any $m\geq 1$,
\[ \rho_m=\max\{C_1 a_{m-1}^{\frac{1}{n(n-1)}}, 2^{-m-1}\}, \quad r_m=r_{m-1}-\rho_m, \quad a_m=\mu(B_{r_m}).\]
Take $R_1=(r_m+r_{m-1})/2$ and $R_2 = r_{m-1}$, so that 
\[
R_2-R_1 = \rho_m/2 \geq (C_1/2) a_{m-1}^{\frac{1}{n(n-1)}} \geq (C_1/4) \mu(B_{R_2})^{\frac{1}{n(n-1)}}.
\]
Thus, we may apply \eqref{eq:MCB}, obtaining
\begin{align}\label{eq:anteMS}
 \mu\left(B_{(r_m+r_{m-1})/2}\right) \rho_m^{-1} &+\int_{F\cap B_{(r_m+r_{m-1})/2}}|H| \, d\m \leq a_{m-1}\rho_m^{-1}  + C_1(\rho_m/2)^{-2} a_{m-1}^{n/(n-1)} \notag \\
& \leq a_{m-1}\rho_m^{-1} + (\rho_m^{-1} a_{m-1}) (4C_1 a_{m-1}^{1/(n-1)}\rho_m^{-1}) \leq c a_{m-1}\rho_m^{-1} \\
& \leq C a_{m-1}2^m, \notag
%C4^{(n+1)m} a_{m-1}.
\end{align}
since $4C_1 a_{m-1}^{1/(n-1)}\rho_m^{-1} \leq 4 a_{m-1}^{1/n} \leq 1$ when $a_{m-1}\leq \vare$ is small enough.
Let $\phi\in C^1_c(B_{(r_m+r_{m-1})/2})$ be a cut-off function, such that $0\leq \phi\leq 1$, $\phi=1$ in $B_{r_m}$ and $|\nabla \phi| \leq c\rho_m^{-1}$ in $B_1$. Plugging it in the Michael-Simon inequality (\cite{M-Simon}), 
\[
\left(\int_F \phi^{\frac{n-1}{n-2}} \, d\m\right)^{\frac{n-2}{n-1}} \leq C_0 \int_F \left( |\nabla_F \phi| + |H|\phi\right) \, d\mu
\]
and using \eqref{eq:anteMS}, we get %a dimensional constant $A=A(n)> 1$ such that
\begin{align}\label{eq:DG}
 a_m & \leq \int_F \phi^{\frac{n-1}{n-2}} \, d\m \leq C\left(\mu\left(B_{r_m+r_{m-1})/2}\right) \rho_m^{-1} +\int_{F\cap B_{(r_m+r_{m-1})/2}}|H| \, d\m \right)^{\frac{n-1}{n-2}} \notag \\
& \leq c \left(a_{m-1}2^m\right)^{\frac{n-1}{n-2}}. %\leq \left(A^m a_{m-1}\right)^{\frac{n-1}{n-2}}. 
\end{align}
We now argue by induction that 
\begin{equation}\label{eq:decay}
a_m\leq \vare \g^{-m}
\end{equation}
for $\g=2^{n-1}>1$ and $\vare\leq \vare_2(n)$ small enough. Set $\k:=\frac{n-1}{n-2}$. We see that \eqref{eq:decay} holds for $m=0$ and assume it is true for $m-1$. Using the recurrence inequality \eqref{eq:DG}, we obtain
\begin{align*}
a_m\leq c 2^{m\k} a_{m-1}^\k \leq c 2^{m\k} \vare^\k \g^{(-m+1)\k} = \vare \g^{-m} (c \vare^{\k-1} \g^\k) \left(2^\k \g^{-(\k-1)}\right)^m \leq \vare \g^{-m},
\end{align*}
if we choose $\g = 2^{\k/(\k-1)}= 2^{n-1}>1$ and $\vare \leq \vare_2(n):=(c\g^\k)^{-1/(\k-1)}$. Employing  \eqref{eq:decay}, we now estimate that for all $m\in \N$
\[
C_1 a_{m-1}^{\frac{1}{n(n-1)}} \leq C_1 \vare^{\frac{1}{n(n-1)}} 2^{-\frac{m-1}{n}},% \quad \text{provided} \quad \vare \leq \vare_3(n):=(8 C_1)^{-(n-1)}.
\]
whence
\[\rho_m \leq \max\{2^{-m-1}, \vare^{\frac{1}{n(n-1)}} 2^{-\frac{m-1}{n}} \} \quad \text{and} \quad r_m=1-\sum_{k=1}^{m}\rho_k\geq 1/2,\]
whenever $\vare\leq \vare_3(n)$ is small enough. Therefore, if we choose $\vare(n):=\min\{\vare_1, \vare_2, \vare_3\}$, we obtain the desired
\[ \mu(B_{1/2})\leq \lim_{m\to+\infty}\mu(B_{r_m})=\lim_{m\to\infty}a_m=0. \qedhere\]
\end{proof}

\section{Proof of Theorem \ref{thm:bernstab}}\label{sec:proofstab}

We will start this section by establishing the Bernstein type statement \eqref{eq:FBBern} for $n=2$:  the only entire stable classical solutions to the one-phase FBP in the plane are, up to rigid motion, the \emph{one-plane} $U(x_1,x_2)=x_2^+$ and the \emph{two-plane} solutions $U(x_1, x_2) = x_2^+ + (x_2 + a)^-$, $a> 0$.

\begin{thm}\label{prop:stab} Let $U:\R^2\to [0,\infty)$ be an entire stable classical solution to the one-phase FBP \eqref{eqn}. Then $u$ is a one-plane or a two-plane solution.
\end{thm}
\begin{proof} For any $\eta \in C^{\infty}_0(\R^n)$ we plug in the test function $\phi = |\nabla U|\eta$ in the stability inequality, obtaining
\[
\int_F H \eta^2 \, d\m \leq \int_\O \left(|\nabla |\nabla U||^2 \eta^2 + |\nabla U|^2 |\nabla \eta|^2 + \frac{1}{2}\nabla |\nabla U|^2 \cdot \nabla \eta^2\right)\, dx.
\]
Since $\frac{1}{2}(|\nabla U|^2)_\nu = H$ on the free boundary $F$ while $\D |\nabla U|^2 = 2 |\nabla ^2 U|^2$ in the positive phase $\O=\{U>0\}$, an integration by parts yields
\[
\int_{\O} \left(|\nabla^2 U|^2  - |\nabla |\nabla U||^2\right)\eta^2 \, dx \leq \int_\O |\nabla U|^2 |\nabla \eta |^2 \, dx.
\] 
Combining the fact \eqref{eq:Lip} that $U$ is Lipschitz continuous (see also Proposition \ref{prop:gradient}),
\[
|\nabla U|\leq C \quad \text{in } \R^2,
\]
with the result of Proposition \ref{lem:harmonic}, gives us the estimate on the Hessian % $|\nabla^2 u|^2  - |\nabla |\nabla u||^2\geq c |\nabla^2 u|^2$ a.e.\ in $\O$  gives the estimate
\begin{equation}\label{prop:stab:main}
\int_\O |\nabla^2 U|^2 \eta^2 \, dx \leq C \int_{\O} |\nabla \eta|^2 \, dx. 
\end{equation}
By plugging in \eqref{prop:stab:main} a standard \emph{logarithmic cut-off} function $\eta=\eta_R$
\[
\eta_R(x):= \begin{cases} 1 & \text{for } |x|\leq R, \\ 
2-(\log|x|)/\log R & \text{for } R<|x|\leq R^2, \\
0 & \text{for } |x|>R^2,
\end{cases}
\]
we now obtain
\[
\int_{\O \cap B_R} |\nabla^2 U|^2 \, dx \leq \frac{C}{\log^2 R} \int_{\O\cap (B_{R^2}\setminus B_R)} \frac{1}{|x|^2} \, dx  \leq \frac{C'}{\log R}.
\]
Taking $R\to \infty$, we conclude that $|\nabla^2 U|^2 = 0$ in $\O$ so that $u$ is a linear function in each connected component of $\O$. There are only two possibilities, up to Euclidean isometry: $U=x_2^+$, or $U=x_2^+ + (x_2 + a)^-$, for some $a>0$.
\end{proof}

We now borrow a strategy from minimal surface theory (see \cite[Chapter 3]{white2013lectures}) to prove that the rigidity of the Bernstein type problem for \eqref{eqn} in dimension $n$
%\begin{equation}\label{eq:FBBern}
%\text{if } U \text{ is a stable solution of \eqref{eqn} in } \R^n \quad \Longrightarrow \quad U \text{ is a one/two-plane solution} 
%\end{equation}
entails interior bounds for the curvature of the free boundary in the local problem.

\begin{proof}[Proof of Theorem \ref{thm:bernstab}]
Assume that the statement of the theorem is false. Then there exists a sequence of counterexamples $\{u_k\}_{k\in \N}$ defined in $B_1$, whose free boundaries $F(u_k)$  have second fundamental forms $A_k$ that satisfy:
\[
\rho_k:=\max_{p\in F(u_k)\cap \overline{B_{1/4}}} |A_k(p)|\text{dist}(p, \de B_{1/4}) \nearrow \infty, \quad \text{as } k\to \infty.
\]
Let $p_k\in F(u_k)\cap B_{1/4}$ be a point where the maximum $\rho_k$ is attained and define the sequence of rescaled solutions
\[
\tilde{u}_k(x):= |A_k(p_k)| u_k (p_k + x /|A_k(p_k)|) \quad \text{for } x \in (B_{1/4} - p_k)|A_k(p_k)|\supseteq B_{\rho_k}, 
\]
which are uniformly Lipschitz continuous and uniformly nondegenerate in $B_{\rho_k}$, the latter on account of Theorem \ref{thm nondegeneracy}. 
Furthermore, the free boundary $F(\tilde{u}_k)$ of $\tilde{u}_k$ has second fundamental form $\tilde{A}_k$ that satisfies 
\begin{align}
|\tilde{A}_k(0)| =  1 & \quad \text{and} \label{eq:normalizedcurv} \\
|\tilde{A}_k(p)|\text{dist}(p, \de B_{\rho_k}) \leq \rho_k & \quad \text{for all } p\in F(\tilde{u}_k)\cap B_{\rho_k}. \notag
\end{align}
This means that for each fixed $r>0$
\begin{equation}\label{thm:bernstab:eqcurv}
\sup_{p\in F(\tilde{u}_k)\cap B_r} |\tilde{A}_k(p)| \leq \frac{\rho_k}{\rho_k - r} \to 1 \quad \text{as } k\to \infty. 
\end{equation}
Now, according to Proposition \ref{prop:limitofsolns}, up to taking a subsequence, $\{\tilde{u}_k\}$ converges uniformly on compacts to a globally defined \emph{viscosity} solution $U$, while both 
\begin{equation}\label{thm:bernstab:hausdorff}
F(\tilde{u}_k) \to F(U) \quad \text{and} \quad \overline{\{\tilde{u}_k>0\}} \to \overline{\{U>0\}}
\end{equation}
converge locally in the Hausdorff distance. Furthermore, by Proposition \ref{prop:gradient}, we know that
\begin{equation}\label{thm:bernstab:grad}
|\nabla U| \leq 1 \quad \text{in} \quad \{U>0\}.
\end{equation}
 We will argue that $U$ is, in fact, an entire \emph{classical stable} solution of \eqref{eqn}. 

Let $q\in F(U)$. We aim to show that near $q$ the free boundary $F(U)$ is the graph of a smooth function, separating positive from zero phase. Let $q_k\in F(\tilde{u}_k)$ be a sequence of points that converge to $q$ as $k\to\infty$, and let $\nu_k$ be the inner unit normal vector to $\{\tilde{u}_k>0\}$ at $q_k$. By possibly taking a further subsequence and rotating the coordinate axes, we may assume that $\nu_k$ converge to $e_n$. 
Now, because of the uniformly bounded (on compacts) curvature \eqref{thm:bernstab:eqcurv} of the free boundaries $F(\tilde{u}_k)$, there is a small absolute constant $0<c<1$ such that for all large $k$, the connected component $\mathcal{C}_k$ of $[B_c(q)]^+(\tilde{u}_k)$, whose boundary contains $q_k$, is a supergraph
\begin{equation}\label{thm:bernstab:eq:supergraph}
\mathcal{C}_k = \{y\in B_{c}(q): y=q+(x',x_n),~ x_n> f_k(x')\},
\end{equation}
of a function $f_k:B'_c\subset \R^{n-1}\to \R$ with uniformly bounded $C^2$ norm (Lemma \ref{lem:unifcurv}):
\begin{equation}\label{thm:bernstab:eq:AA}
\|f_k\|_{C^2(B'_{c})}\leq 1/c \quad \text{and} \quad \lim_{k\to\infty} f_k(0') = 0.
\end{equation}
By Arzela-Ascoli, the sequence $\{f_k\}$ subconverges in $C^{1,\a}(B_c')$ to some $f\in C^2(B_c')$ with $f(0')=0$, so that by \eqref{thm:bernstab:hausdorff}, the set
\[
\mathcal{C} :=  \{y\in B_{c}(q): y=q+(x',x_n),~ x_n> f(x')\},
\]
is a connected component of $[B_c(q)]^+(U)$. By the classical regularity theory for the one-phase FBP \cite{KindNiren}, $f$ is in fact smooth,
\begin{equation}\label{thm:bernstab:eq:conv}
f_k \to f \quad \text{in } C^m(B_c') \quad \text{for any } m\in \N,
\end{equation}
and  $U_+:=U\chrc{\mathcal{C}}$ is a classical solution of \eqref{eqn} in $B_c(q)$. In particular, when $q=0$ is the origin, the normalized curvature condition \eqref{eq:normalizedcurv} implies that $F(U_+)$ has second fundamental form $\tilde{A}$ at $0$ of unit magnitude:
\begin{equation}\label{thm:bernstab:eq:limcurv}
|\tilde{A}(0)|=1.
\end{equation}
%Note that one side effect of \eqref{thm:bernstab:eq:limcurv} is that $|\nabla U|<1$ in at least one connected component of the positive phase of $U$, for otherwise $F(U)$ is flat (see Proposition \ref{prop:gradient}). 

To finish the proof that $U$ itself is a classical solution in $B_c(q)$, it remains to verify that there are no two connected components of $[B_c(q)]^+(U)$ touching at $q$. Assume, to the contrary, that there is a separate connected component $\tilde{\mathcal{C}}$ of $[B_c(q)]^+(U)$, such that $q\in \de\tilde{\mathcal{C}}$.  By the argument from the previous paragraph, $U_-:=U\chrc{\mathcal{\tilde{C}}}$ is a classical solution of \eqref{eqn} in $B_c(q)$, as well. Denote by $H_+$  and $H_-$ the mean curvature of $F(U_+)$ and $F(U_-)$ with respect to the inner unit normal $\nu_+$ to $\mathcal{C}$ and $\nu_-$ to $\tilde{\mathcal{C}}$, respectively. Observe that in at least one of $\mathcal{C}$ or $\tilde{\mathcal{C}}$ we have the strict inequality $|\nabla U_\pm|<1$, for otherwise, Proposition \ref{prop:gradient} would suggest that $U(x)=|(x-q)_n|$, contradicting the nontrivial curvature condition \eqref{thm:bernstab:eq:limcurv} at $0$. Let us therefore assume that 
\begin{equation}\label{thm:bernstab:assump}
|\nabla U_-|<1 \quad \text{in } \tilde{\mathcal{C}}. 
\end{equation}
Now, because of \eqref{thm:bernstab:grad}, we have
\begin{equation}\label{thm:bernstab:meancurv}
H_{\pm} = \text{div}\left(\frac{\nabla U_\pm}{|\nabla U_\pm|}\right) = - \frac{\de_{\nu_\pm}|\nabla U_\pm|^2}{2|\nabla U_\pm|^2} = -\frac{1}{2}\de_{\nu_\pm}|\nabla U_\pm|^2 \geq 0.
\end{equation}
Moreover, via the Hopf Lemma, \eqref{thm:bernstab:assump} entails that 
\begin{equation}\label{thm:bernstab:Hmin}
H_- >0.
\end{equation}
On the other hand, since $F(U_+)$ and $F(U_-)$ touch at $q$ with normal unit vectors $\nu_+(q)=e_n=-\nu_{-}(q)$, we have the comparison
\begin{equation}\label{thm:bernstab:eq:comp}
-H_-(q) \geq  H_+(q). 
\end{equation}
Combining \eqref{thm:bernstab:meancurv}, \eqref{thm:bernstab:Hmin} and \eqref{thm:bernstab:eq:comp} yields the impossible
\[
0>-H_-(q) \geq  H_+(q)\geq 0. 
\]

We conclude that $U$ is an entire classical solution of the one-phase FBP. By the smooth local convergence of the free boundaries \eqref{thm:bernstab:eq:conv}, we also see that for any fixed $\phi\in C^\infty_0(\R^n)$
\[
\int_{F(\tilde{u}_k)} H_k \phi^2 d\H^{n-1}  \to \int_{F(U)} H \phi^2 \, d\H^{n-1} \quad \text{as } k\to\infty,
\]
where $H_k$ and $H$ denote the mean curvature of $F(\tilde{u}_k)$ and $F(U)$, respectively. Since by Proposition \ref{prop:limitofsolns}, 
\[
\chrc{\{\tilde{u}_k>0\}} \to \chrc{\{U>0\}} \quad \text{in } L^1_\text{loc}, \quad \text{as } k\to\infty,
\]
we also have that
\[
\int_{\{\tilde{u}_k>0\}} |\nabla \phi|^2 \, dx \to \int_{\{U>0\}} |\nabla \phi|^2 \, dx , \quad \text{as } k\to\infty.
\]
By taking the limit as $k\to \infty$, we see then that the stability of $\tilde{u}_k$ entails the stability of $U$.  
%\[
%\int_{F(U)} H \phi^2 \, d\H^{n-1} - \int_{\{U>0\}} |\nabla \phi|^2 \, dx  = \lim_{k\to\infty} \left(\int_{F(\tilde{u}_k)} H_k \phi^2 \, d\H^{n-1} - \int_{\{\tilde{u}_k>0\}} |\nabla \phi|^2 \, dx\right) \leq 0. 
%\]

However, the rigidity statement \eqref{eq:FBBern} now says that $F(U)$ is flat, which contradicts the fact \eqref{thm:bernstab:eq:limcurv} that $F(U)$ has nontrivial curvature at the origin.  The proof is complete.

 %or a ``double" supergraph
%\begin{equation}\label{thm:bernstab:eq:2supergraph}
%B_{c}^+(\tilde{u}_k) = \{x\in B_{2c}: x_n> f^+_k(x')\}  \cup \{x\in B_{2c}: x_n< f^-_k(x')\},
%\end{equation}
%for some pair of functions $f_k^\pm:B'_c\subset \R^{n-1}\to \R$ with $f_k^-<f_k^+$,  $\|f_k^\pm\|_{C^2(B'_{c})}\leq 1/c$ and $\lim_{k\to\infty} f_k^+(0') = 0$.

\end{proof}

\appendix 
\section{Auxiliary results}
We recall the notion of a \emph{viscosity solution} of \eqref{eqn} (see \cite{CafSalsa}).  First we define viscosity super- and subsolutions.
\begin{defi}  \label{def:visc.super} A viscosity supersolution of \eqref{eqn} in a domain $D\subseteq \R^n$ is a non-negative
function $w\in C(D)$ such that $\D w \le 0$ in $D^+(w)$ and 
for every $x_0\in F(w)$ with a tangent ball $B$ from the positive side
($x_0\in \de B$ and $B\subset D^+(w)$), there is $\a \le 1$ such that
\[
u(x) = \a  \langle x-x_0, \nu\rangle^+ + o(|x-x_0|) 
\]
as $x\to x_0$ non-tangentially in $B$, with $\nu$ the inner normal to $\de B$ at $x_0$. 
\end{defi}
%In other words, $u$ is superharmonic in its positive phase and has slope less than or equal to one on the boundary at points with interior tangent balls. 

\begin{defi} \label{def:visc.sub}   A viscosity subsolution of \eqref{eqn} in a domain $D\subseteq \R^n$ is a non-negative
function  $w\in C(D)$  such that $\D w \ge 0$ in $D^+(w)$ and 
for every $x_0\in F(w)$ with a tangent ball $B$ in the zero set
($x_0\in \de B$ and $B\subset \{w=0\}$), there is $\a \ge 1$ such that
\[
u(x) = \a  \langle x-x_0, \nu\rangle^+ + o(|x-x_0|) 
\]
as $x\to x_0$ non-tangentially in $B^c$, with $\nu$ the outer normal to $\de B$ at $x_0$. 
\end{defi}

A {\em viscosity solution} in $D$ is a function that is both a supersolution and a subsolution in
the sense above. 

The class of viscosity solutions is particularly well suited for taking uniform limits. We quote the following result 

\begin{lem}(\cite[Lemma 4.4]{JK16})\label{lem:limitvisco} Let $u_k\in C(D)$ be a sequence of viscosity solutions of \eqref{eqn} in $D$ such that $u_k\rightarrow u$ uniformly and $u$ is Lipschitz continuous.
Then $u$ is also a viscosity supersolution of \eqref{eqn} in
$D$. If, in addition, $\overline{D^{+}(u_k)} \rightarrow
\overline{D^{+}(u)}$ locally in the Hausdorff distance, then
$u$ is a viscosity subsolution, as well.
\end{lem}

As a corollary, we have the following well known result (see \cite[Lemma 1.21]{CafSalsa} or \cite[Proposition 4.2]{JK16} for the proof) describing uniform limits of Lipschitz continuous and non-degenerate viscosity solutions. 
\begin{prop}\label{prop:limitofsolns} Let $D\subseteq \R^n$ be a domain and $\{u_k\} \subset C(D)$ be a sequence of viscosity solutions of \eqref{eqn} in $D$ which satisfies
\begin{itemize}
%\item $0 \in F(u_k)$;
\item (Uniform Lipschitz continuity) There exists a constant C, such that $$\|\nabla u_k\|_{L^{\infty}(D)} \leq C;$$
\item (Uniform non-degeneracy) There exists a constant c, such that $$\dashint_{\de B_r(x)} u_k \, d\H^{n-1} \geq cr$$ 
for every $B_r(x) \subseteq
D$, centered at a free boundary point $x\in F(u_k)$.
\end{itemize}
Then any limit $u \in C(\overline{D})$ of a uniformly
convergent on compacts subsequence $u_k \to u$ satisfies
\begin{enumerate}
%\item $u_k \to u$ uniformly on compact subsets of $\R^2$;
\item $\overline{D^+(u_k)} \rightarrow \overline{D^+(u)}$ and $F(u_k) \rightarrow F(u)$ locally in the Hausdorff distance;
\item $1_{\{u_k>0\}} \rightarrow 1_{\{u>0\}}$ in $L^1_{\text{loc}}(D)$;
%\item $\nabla u_k \rightarrow \nabla u$ a.e. in $D$.
\end{enumerate}
Moreover, $u$ is a Lipschitz continuous, non-degenerate viscosity
solution of \eqref{eqn}.
\end{prop}

In the next proposition we establish the fact that globally defined viscosity solutions of the one-phase FBP have a gradient bounded by $1$.

\begin{prop}\label{prop:gradient} Let $u:\R^n\to [0,\infty)$ be an entire viscosity solution of the one-phase FBP \eqref{eqn}, with $F(u)\neq \emptyset$, and let $\O:=\{x\in \R^n: u(x)>0\}$ denote the positive phase of $u$. Then 
\[
|\nabla u(x)| \leq 1 \quad \text{for all } x\in \O.
\] 
Furthermore, if $|\nabla u(x_0)| = 1$ for some $x_0\in \O$, then $|\nabla u (x)|\equiv 1$ in the connected component $\mathcal{C}$ of $\O$, containing $x_0$, so that $u|_{\mathcal{C}}(x) = (x-p)\cdot e$ for some $p\in \R^n$ and unit vector $e\in \R^n$.
\end{prop}

\begin{proof}
We sketch out the argument from \cite[Proposition 2.1]{KWang-MorseIndex-FB}, written for a closely related problem. 

First, we show that there is a dimensional constant $C>0$ such that $$|\nabla u|\leq C \quad \text{in } \O.$$ Pick any $x\in \O$. By rescaling and recentering, we may assume that $x=0$ and $d(x, F(u))=1$. By applying the Harnack inequality plus gradient estimates to $u$ in $B_1$, it suffices to show that $u(0)\leq C_0$. Let $p\in F(u)\cap \de B_1$ and note that $B_1$ is a ball touching $F(u)$ from the positive side. 

By the Harnack inequality, we know that $u\geq cu(0)$ in $B_{1/2}$. Consider the harmonic function $h$ in the annulus $B_1\setminus B_{1/2}$, having boundary values $h=cu(0)$ on $\de B_{1/2}$ and $h=0$ on $\de B_1$. Then $h_\nu(p) = c_0 u(0)$, where $\nu$ denotes the inner unit normal to $\de B$ at $p$. On the other hand, the maximum principle implies $h\leq u$ in $B_1\setminus B_{1/2}$, so that the Hopf Lemma in conjunction with the viscosity supersolution property yield the desired bound
\[
1 \geq u_\nu(p) \geq h_\nu(p) = c_0 u(0).
\] 

Therefore, the supremum $L:=\sup_{\O} |\nabla u|\leq C$ is positive and finite. Let $x_k\in \O$ be a sequence of points such that $|\nabla u(x_k)| \to L$ as $k\to\infty$.  Define the folowing rescales of $u$
\[
v_k(x):=d_k^{-1} u(x_k + d_k x), \quad \text{where } d_k:=d(x_k, F(u)). 
\]
Then the $v_k$ satisfy: $\|\nabla v_k\|\leq L$, the distance $d(0,F(v_k))=1$ and $|\nabla v_k(0)|=|\nabla u(x_k)| \to L$ as $k\to \infty$. Thus, the sequence of uniformly Lipschitz continuous, viscosity solutions $\{v_k\}$ subconverges uniformly on compact subsets of $\R^n$ to the globally defined, Lipschitz continuous function $v$, which is a harmonic in its positive phase $\tilde{\O}:=\{v>0\}$, satisfies $\|\nabla v\|_{L^\infty}\leq L$, and which is a viscosity \emph{supersolution} on account of Lemma \ref{lem:limitvisco}. Furthermore, since $v_k\to v$ uniformly on $\overline{B_1}$ and $v_k-v$ is harmonic in $B_1$, we have
\[
|\nabla v(0) - \nabla v_k(0)| \leq  C  \|v_k - v\|_{C(B_1)} \to 0 \quad \text{as } k\to\infty, 
\] 
so that $|\nabla v(0)|=L$. 

However, since $\D |\nabla v|^2 =2|\nabla^2 v|^2 \geq 0$ in $\tilde{\O}$, we know that $v$ is subharmonic in $\tilde{\O}$, so that the strict maximum principle implies that $|\nabla v|\equiv L$ in the connected component $\mathcal{C}$ of  $\tilde{\O}$ containing $0$. Therefore, $2|\nabla^2 v|^2 = \D |\nabla v|^2 = 0$ in $\mathcal{C}$, so that $v$ is a linear function in $\mathcal{C}$ with slope $L>0$. This means that the component $\mathcal{C}$ is actually a half-space and for some $p\in \R^n$, and unit vector $e\in \R^n$, we have
\[
v(x) = L(x-p)\cdot e \quad \text{for all } x\in \mathcal{C}.
\]
We now infer that $L\leq 1$ from the fact that $v$ is a viscosity supersolution to \eqref{eqn}.

To establish the remaining part of the proposition (regarding the possibility of $|\nabla u(x_0)|=1$), we use the verbatim argument from the paragraph above. 
\end{proof}

We end the appendix with the following lemma.
\begin{lem}\label{lem:unifcurv} Let $u$ be a classical solution to \eqref{eqn} in $B_1\subset \R^n$ which is Lipschitz continuous and nondegenerate (with universal constants) in $B_1$. Assume that $0\in F(u)$ and that the second fundamental form $A$ of $F(u)$ is bounded
\[
|A(p)| \leq C \quad \text{for all } p\in F(u),
\]
by some absolute constant $C>0$. Then there exists a constant $c\in (0,1)$ such that, in a suitable Euclidean coordinate system, the connected component $\mathcal{C}$ of $B_c^+(u)$, whose boundary contains $0$, is the supergraph:
\begin{equation}\label{lem:unifcurv:supergraph}
\mathcal{C}= \{x=(x',x_n)\in B_{c}: x_n> f(x')\},
\end{equation}
for some $f:B_{c}'\to \R$ with $\|f\|_{C^2(B_c')}\leq 1/c$.
\end{lem}
\begin{proof}
Since the curvature of $F(u)$ is bounded by an absolute constant, there is a small absolute constant $c>0$, such that the component $F$ of $F(u)\cap B_{c}$ containing $0$, is given by a graph 
\[
F = \{(x',x_n)\in B_{c}: x_n= f(x')\},
\]
with $\|f\|_{C^2(B_c')}\leq 1/c$. Let us show that, by possibly reducing the constant $c$, the connected component $\mathcal{C}$ of $B_c^+(u)$ bordering the origin, is the supergraph \eqref{lem:unifcurv:supergraph}. 

Assume that this last statement is false. Then there exist a sequence $c_k\to 0$ and a sequence of counterexamples $\{u_k\}$ such that the component  $\mathcal{C}_k$ of $B_{c_k}^+(u_k)$ bordering $0$ has at least two free boundary connected components: the connected component $F_k$ of $0$ in $F(u_k)\cap B_{c_k}$, and another component $\tilde{F}_k$. Choose $\tilde{F_k}$ to be the closest such component to $F_k$ and set $d_k:=\text{dist}(F_k,\tilde{F}_k)\leq c_k$. 
Consider now the rescaled solutions
\[
v_k(x):=d_k^{-1} [u_k\mathcal{C}_k](d_k x) \quad \text{in } B_{d_k^{-1}c_k}\supseteq B_1,
\]
and denote by $G_k:=d_k^{-1}F_k$, $\tilde{G}_k:=d_k^{-1}\tilde{F}_k$. 
We have that
\begin{equation}\label{eq:othercomp}
\text{dist}(G_k, \tilde{G}_k) = 1, 
\end{equation}
while the curvature $$|A_{G_k}|\leq d_k \sup |A_{F(u_k)}| \leq c_k C \to 0.$$ Now, according to Proposition \ref{prop:limitofsolns} the uniformly Lipschitz continuous and nondegenerate $v_k$ converge to the limit $v=x_n^+$ uniformly on compacts, with $G_k$ converging to a subset $G\subseteq\{x_n=0\}$. On the other hand, by the Hausdorff convergence of the free boundaries, \eqref{eq:othercomp} suggests that $F(v)$ also has a component sitting at a unit distance  away from $G$. This yields a contradiction and the proof of the lemma is complete. \qedhere

\end{proof}

\bibliography{Nondeg_bib}
\end{document}